\documentclass[11pt]{article}

\usepackage{amsfonts,amsmath, amssymb, amsthm,latexsym,color,epsfig,mathrsfs,enumitem}
\newtheorem{thm}{Theorem}

\newtheorem{lem}[thm]{Lemma}

\newcommand{\A}{\mathcal{A}}

\newcommand{\pe}{\partial _e}
\DeclareMathOperator{\dir}{Dir}

\newcommand{\eps}{\varepsilon}
\newcommand{\es}{\emptyset}
\newcommand{\aA}{\alpha}

\title{A stability result for the cube edge isoperimetric inequality}
\author{Peter Keevash\thanks{Mathematical Institute, University of Oxford, Oxford, UK. 
E-mail: keevash@maths.ox.ac.uk.\newline \hspace*{1.5em}
Research supported in part by ERC Consolidator Grant 647678.}
\and Eoin Long\thanks{Mathematical Institute, University of Oxford, Oxford, UK. 
E-mail: long@maths.ox.ac.uk.}}
\date{\today}

\begin{document}

\maketitle
\abstract{
We prove the following stability version 
of the edge isoperimetric inequality for the cube:
any subset of the cube with average boundary degree
within $K$ of the minimum possible
is $\eps$-close to a union of $L$ disjoint cubes,
where $L \le L(K,\eps)$ is independent of the dimension.
This extends a stability result of Ellis,
and can viewed as a dimension-free version 
of Friedgut's junta theorem.}

\section{Introduction}

The edge isoperimetric inequality 
is a fundamental result in Extremal Combinatorics 
concerning the distribution of edges in the cube. 
The $n$-\emph{cube} $Q_n$ is the graph on vertex set $\{0,1\}^n$ 
in which vertices are adjacent if they differ in a single coordinate. 
The \emph{edge boundary} of a set $\A \subset V(Q_n)$ 
is the set of edges $\pe (\A ) \subset E(Q_n)$ that leave $\A $, i.e.\ $\pe (\A ) = 
\{xy \in E(Q_n): x \in \A , y \notin \A \}$.
A tight lower bound on $|\pe (\A )|$ was given 
by Bernstein \cite{bernstein}, Harper \cite{harper},
Hart \cite{hart} and Lindsey \cite{lindsey},
who proved that the extremal sets are initial segments 
of the `binary ordering' on $Q_n$ (see also Chapter 16 of \cite{bol}).
In particular, the following bound is tight
when $|\A| = 2^d$ for some $d \in {\mathbb N}$
(take $\A$ to be the vertices of a $d$-dimensional subcube).
\begin{thm}
	\label{thm: edge iso}
	Every $\A \subset V(Q_n)$ satisfies $|\pe (\A)| \geq 
	|\A | \cdot \log _2\big ( 2^n / |\A|\big )$.
\end{thm}
 
The next natural question is to understand the structure 
of subsets in the cube for which the inequality 
in Theorem \ref{thm: edge iso} is close to an equality:
must they be close to an extremal example?
Indeed, for any problem in Extremal Combinatorics,
the study of this `stability' question often leads
to a deeper understanding of the original question.
The following stability version of Theorem \ref{thm: edge iso}
was obtained by Ellis \cite{Ellis}, 
solving a conjecture of Bollob\'as, Leader and Riordan. 

\begin{thm}
	\label{thm: Ellis}
	There is $\varepsilon _0>0$ so that for $0 \leq \varepsilon \leq \varepsilon _0$, the following holds. 
	Suppose $\A \subset V(Q_n)$ with $|\pe (\A )| \leq 
	|\A | \big ( \log _2(2^n / |\A |) + \varepsilon )$. 
	Then there is a subcube $C$ of $Q_n$ with $|\A \triangle C| \leq 
	\frac {3 \varepsilon }{\log _2 (\varepsilon ^{-1})}|\A |$.
\end{thm}

\noindent This result has recently been refined 
by Ellis, Keller and Lifshitz \cite{EKL}
in the regime of extremely close approximation: 
they proved that a set $\A $ whose edge boundary 
is within an additive constant $d$ of the minimum possible
is $O(d)$-close to (an isomorphic copy of)
the unique isoperimetric set.

It is generally more challenging to obtain any 
structural information in an extremal problem
as the distance from the extremum increases.
Kahn and Kalai \cite{KK} made a series of compelling
conjectures around the theme of thresholds of monotone properties,
some of which explore a potential connection
with the stability problem for the
edge-isoperimetric inequality.

One such conjecture, in a strengthened form
proposed by Ellis in \cite[Conjecture 3.3]{Ellis},
suggests that small edge-boundary should imply
some correlation with a large subcube.
Concretely, for any $K>0$ there are $K',\delta$ so that
if $\A \subset \{0,1\}^n$ with $|\A | = \alpha 2^n$ 
and $|\partial _e(\A )| \leq K |\A | \log _2 \aA^{-1}$
then there should be a subcube $C$ of $\{0,1\}^n$ 
of codimension at most $K'\log _2 \aA^{-1}$ 
with $|\A \cap C| \geq (1 + \delta )\alpha |C|$. 
Kahn and Kalai proposed this conjecture
in the special case of monotone properties,
but in the more general setting
of biased measures on the cube.
Kahn and Kalai further conjecture
(see \cite[Conjecture 4.1(b)]{KK},
again for monotone properties)
that such $\A$ must be close
to a union of at most $\alpha ^{-K'}$ cubes.

A weaker form of the latter conjecture follows from a result of Friedgut \cite{Friedgut} in the `dense' regime.
\begin{thm}
	\label{thm: Friedgut}
	Let $K, \varepsilon >0$. Suppose that ${\cal A} \subset \{0,1\}^n$ 
	with $|\partial _e({\cal A})| \leq K2^n$. Then there are 
	disjoint cubes $C_1,\ldots ,C_L$ with 
	$|\A \triangle (C_1\cup \cdots \cup C_{L})| \leq \varepsilon 2^n$, where 
	$L \leq L(K,\varepsilon ) = 2^{2^{C(K/\varepsilon )}}$ for some 
	constant $C>0$.
\end{thm}
\noindent \emph{Remark:} Friedgut actually proved that given such $\A $ there is set $S \subset [n]$ with $|S| \leq  D:= 2^{C(K/\varepsilon )}$ so that ${\cal A}$ is $\varepsilon$-approximated by disjoint cubes, all of whose fixed coordinate sets lie in $S$ (often stated as ${\cal A}$ is $\varepsilon $-close to a $D$-Junta). Theorem \ref{thm: Friedgut} is an immediate consequence of this.

Our main result gives an analogue of Friedgut's theorem that also applies in the sparse regime.

\begin{thm}
	\label{dimension free stability}
Let $K, \varepsilon >0$. Suppose that $\A \subset V(Q_n)$ with $|\partial _e(\A )| \leq |\A | \big (\log _2(2^n/|\A |) + K \big )$. Then there are disjoint cubes $C_1,\ldots ,C_{L}$ with $|\A \triangle (C_1\cup \cdots \cup C_{L})| \leq \varepsilon |\A |$, where $L \leq L(K,\varepsilon ) = 2^{2^{C(K/\varepsilon )^2}}$ for some constant $C>0$.
\end{thm}

\noindent \textit{Remark:} Letting $E(\A )$ denote the set of edges in $\A $, 
i.e.\ $E(\A) = \{xx' \in E(Q_n): x,x' \in \A \}$, Theorem \ref{thm: edge iso} is equivalent to 
$|E(\A)| \leq |\A | (\log _2 |\A |) /2$. In this setting, Theorem \ref{dimension free stability} says that if $\A  \subset \{0,1\}^n$ with $|E(\A )| \geq |\A | (\log _2|\A | - K)/2$ then $\A $ can be $\varepsilon $-approximated by at most $L(K, \varepsilon )$ subcubes. In this sense, Theorem \ref{dimension free stability} gives a `dimension free' stability theorem. 

We conclude this introduction with a brief outline of our argument,
and how the paper will be organised to implement this.
Most of the proof is geared towards showing that $\A$
has a coordinate of significant influence.
This exploits the connection between edge-boundary
and the influences of Boolean functions, which we
will discuss in the next section, together with
two inequalities (due to Talagrand and to Polyanskiy)
that we will use in our proof.
The starting point of our strategy is to choose 
an appropriate partition of the coordinate set,
such that we maintain control on two important quantities:
the constant $K$ appearing in Theorem \ref{dimension free stability}
(which we call the \emph{isoperimetric excess} of ${\cal A}$)
and a certain `mutual information' quantity
(in the sense of information theory).
In section 3 we prove a partitioning lemma
that will enable us to control both these quantities.
The mutual information is then used in section 4
to show that $\A$ is `product-like' in a certain sense. 
The control on the isoperimetric excess will be such that we can 
apply Ellis's theorem to approximate certain sections of $\A$ by cubes,  
provided that they are not too dense. To address the latter point 
(density of sections), in section 5 we apply 
Polyanskiy's hypercontractive inequality to show that $\A$ 
is typically not too dense in random subcubes
(this result can be viewed as a sparse variant of 
the ``It Ain't Over Till It's Over'' conjecture,
proved by Mossel, O'Donnell and Oleszkiewicz \cite{moo}).
The results of the previous sections are combined 
in section 6 in finding a coordinate of significant influence.
This is the main ingredient of an inductive proof of
our main theorem, given in the final section.

\section{Influences of Boolean functions}

Edge boundary has a natural reformulation in terms 
of the analysis of Boolean functions,
which is an active area in its own right,
with many applications to other fields,
including Social Choice and Computational Complexity;
we refer the reader to the book \cite{o'd} for an introduction.
While our approach in this paper will be generally
combinatorial rather analytical, we will require
some auxiliary results obtained by these analytic means.

To discuss this connection we require some notation and terminology.
Given $f: \{0,1\}^n \to {\mathbb R}$, let ${\mathbb E}(f) = 2^{-n} \sum _{x \in \{0,1\}^n} f(x)$ and ${\mathbb V}ar(f) = {\mathbb E}(f - {\mathbb E}(f))^2$, the expectation and variance of $f$ respectively. The function $f$ is said to be Boolean if $f: \{0,1\}^n \to \{0,1\}$. Subsets of $V(Q_n)$ are naturally identified with Boolean functions, where a set $\A \subset \{0,1\}^n$ corresponds to the indicator function $1 _{\A}$, with $1_{\cal A}(x) = 1$ if $x\in \A$ and $0$ otherwise. Given $x \in \{0,1\}^n$ and $i\in [n]$ let $x \oplus e_i$ denote the element of $V(Q_n)$ obtained by changing the $i$th coordinate of $x$. The \emph{influence of $f$ in direction $i$} is defined as $I_i(f)  :=  {|\{x \in \{0,1\}^n: f(x) \neq f(x \oplus e_i)\}|} /{2^n}$.
The \emph{influence of} $f$, denoted $I(f)$, is simply the sum of the individual influences, 
i.e.\ $I(f) = \sum _{i\in [n]} I_i(f)$. Thus $I_i(f)$ denotes the proportion of edges in direction $i$ whose vertices disagree under $f$, and so $I(1_{\cal A}) = |\pe (\A )| / 2^{n-1}$ for all $\A $. Thus any statement regarding the edge boundary of $\A$ is equivalent to a statement on the influence of $1_{\cal A}$.

The notion of influence was first introduced by Ben-Or and Linial  \cite{BL} 
in the context of social choice theory. They conjectured that any Boolean function 
$f:\{0,1\}^n \to \{0,1\}$ with ${\mathbb E}(f) = 1/2$ satisfies $\max _{i\in [n]} I_i(f) = 
\Omega \big ( (\log n) /n \big )$. This was later established by the fundamental KKL theorem of
Kahn, Kalai and Linial \cite{KKL}, who proved that such $f$ satisfy $\sum _{i\in [n]} I_i(f)^2 =
\Omega \big ( (\log^2 n) / n \big )$. The following related inequality, 
that we will use in this paper, was given by Talagrand \cite{Tal}.

\begin{thm}
	\label{thm: Talagrand}
	Any Boolean function $f: \{0,1\}^n \to \{0,1\}$ satisfies 
	$\sum _{i\in [n]} \frac {I_i(f)}{1 - \log _2 I_i(f)} 
	\geq c .{\mathbb V}ar(f)$, where $c>0$ is a constant.
\end{thm}

An important tool in the proof of the KKL theorem (and many results in this area)
is hypercontractivity of the noise operator, due to Bonami \cite{Bon} and
Beckner \cite{Beckner} (see also \cite[Chapter 9]{o'd}).
(An alternative approach based on martingales
and the log-Sobolev inequality for the cube 
was given by Falik and Samorodnitsky \cite{FS} and Rossignol \cite{R}.)
Hypercontractivity will also be important for us in this paper,
via the following estimate for spherical averages due to Polyanskiy \cite{poly}.

For $p \in [1,\infty]$ let $L_p(\{0,1\}^n)$ denote the set of functions $f:\{0,1\}^n \to {\mathbb R}$ equipped with the norm $\|\cdot \|_p$ where $\|f\|_p = (2^{-n} \sum _{x\in \{0,1\}^n} |f(x)|^p )^{1/p}$. For $p = 2$, the space $L_2(\{0,1\}^n)$ also forms a Hilbert space equipped with the usual inner product, given by $\langle f , g \rangle  = \frac {1}{2^n} \sum _{{x} \in \{0,1\}^n} f({x}) g({x})$. Writing $d_H(x,\widetilde x)$ for the Hamming distance between $x$ and $\widetilde {x}$, let $S_{\ell }: L_2(\{0,1\}^n) \to L_2(\{0,1\}^n)$ denote the 
linear operator acting on $f \in L_2(\{0,1\}^n)$ pointwise by 
\begin{equation}
	\label{equation: hypercontractive operator}
	S_{\ell}(f)({x}) = \frac {1}{\binom {n}{\ell } }
				\sum _{\widetilde {x} : d_H({x}, \widetilde {x}) = \ell } f({\widetilde {x}}).
\end{equation}
Polyanskiy proved the following result in \cite{poly} (see Theorem 1, with the remark following it).

\begin{thm}
	\label{thm: Polyansky}
	Let $\ell \in [0,0.15n] \cap {\mathbb N}$. Then for any 
	$f: \{0,1\}^n \to \mathbb R$
		\begin{equation*}
				\|S_{\ell }(f)\| _2 \leq 2^{1/2} \| f\|_{1+{(1-2\ell / n )^2}}.
		\end{equation*}
\end{thm}

While we do not need it for this paper, we should also remark that 
the threshold conjectures of Kahn and Kalai are intimately
connected via Russo's lemma \cite{russo} to the large
literature on influences under $p$-biased measures, 
which can be viewed as a weighted edge boundary
(see e.g.\ \cite{B, FrShThres, FK,H,Tal}).

\section{A partitioning lemma}

In this section we establish some notation
for partitions of the coordinate set
and the corresponding sections of $\A$
that will be used throughout the paper.
We also prove a lemma which shows that
lower-dimensional sections of $\A$
tend to have smaller isoperimetric excess than $\A$,
and also bounds a certain `mutual information' 
that will be used in the next section
to show that $\A$ has an approximate product structure.

Given $x = (x_i)_{i\in [n]} \in \{0,1\}^{n}$ 
and a set $I \subset [n]$ the $I$-\emph{restriction} of $x$ is the vector
$x_I = (x_i)_{i\in I} \in \{0,1\}^{I}$. Given a partition $[n] = I_1 \cup \cdots \cup I_M$ and vectors $x^{(m)} \in \{0,1\}^{I_m}$ for all $m\in [M]$, let 
$x^{(1)}\circ \cdots \circ x^{(M)} \in \{0,1\}^n$ denote the \emph{concatenation} of $x^{(1)},\ldots ,x^{(M)}$, the unique vector $y \in \{0,1\}^n$ with $y_{I_m} = x^{(m)}$ for all $m\in [M]$. Note that $x = x_{I_1} \circ \cdots \circ x_{I_M}$ for all $x \in \{0,1\}^n$.

Let $\dir : E(Q_n) \to [n]$ denote the function with $\dir (xx') = i$ if $x' = x \oplus e_i$. For $I \subset [n]$ and ${\cal A} \subset \{0,1\}^n$ the set $\partial ^I_e(\A) := \big \{xx' \in \partial _e(\A ): \dir (xx') \in I \big \}$ is the $I$-\emph{edge boundary} of $\A$.
Clearly $\partial ^{[n]}_e(\A ) = \partial _e(\A )$. Given a partition $I \cup J = [n]$ and ${y}\in \{0,1\}^J$ the \emph{${y}$-section of $\A $} is the set 
	\begin{equation*}
		\A ^I_{ y} 
				:= 
		\{{z}\in \{0,1\}^{I} : {y} \circ {z} \in \A \} \subset \{0,1\}^{I}.
	\end{equation*} 
The contributions from different sections give $\partial ^{I}_e (\A ) = \bigcup _{{y} \in \{0,1\}^{J} } \partial_e (\A ^I_{y})$.

 Given a set $I \subset [n]$, with 
	complement $J = [n] \setminus  I$, let:
	\begin{itemize}
		\item 
		${\alpha }^{I} = (\alpha ^{I}_{y})$ be 
		the probability distribution on $\{0,1\}^{J}$, with 
		$\alpha ^{I}_{y} = |\A ^{I}_{y}| / |\A |$ for all 
		$y \in \{0,1\}^{J}$;
		\item 
		$|\partial ^I_e ({\A ^{I}_{y}})| = 
		| {\A ^{I}_{y}} | \big ( \log _2 (2^{|I|}/|{\A ^{I}_{y}}|) 
		+ {K^{I}_{y}} \big ) $ for all $y \in \{0,1\}^I$, and set $K^{I} = \sum _{{y}} \alpha ^{I}_{y} K^{I}_{y}$.
	\end{itemize}
Note in particular that $\alpha^\es$ is uniformly distributed on $\A$,
i.e.\ $\alpha^\es(x)$ is $1/|\A|$ if $x \in \A$ or $0$ otherwise.

These section sizes can be naturally reformulated in terms of the following random variables.
Consider selecting $x \in \A $ uniformly at random.
Let ${\bf X}_i$ for $i \in [n]$ denote the random variable ${\bf X}_i(x) = x_i$. 
Write ${\bf X}_I = ({\bf X}_i)_{i\in I}$ for $I \subset [n]$.
Then ${\bf X}_J$ satisfies ${\mathbb P}({\bf X}_J = y) = \alpha^I_y$.

We will see that the entropy of these random variables appears naturally 
in the edge-isoperimetric problem. First we recall some standard definitions
(for an introduction to information theory see the book \cite{ct}).
Given a probability distribution ${\bf p} = (p_{\omega})_{\omega \in \Omega }$ on a finite set $\Omega$, the binary entropy of ${\bf p}$ is given by $H({\bf p}) = - \sum _{\omega \in {\Omega }} p_{\omega } \log _2 p_{\omega }$. Given $\gamma \in [0,1]$ we will also sometimes write 
$H(\gamma )$ for the binary entropy of the probability distribution $\{\gamma , 1- \gamma \}$, i.e.\ $H(\gamma ) = -\gamma \log _2 \gamma - (1-\gamma ) \log _2 (1 - \gamma )$. The entropy of a random variable ${\bf X}$, denoted $H({\bf X})$, taking values in $\Omega $ is the entropy of its probability mass function, i.e.\ $H({\bf X}) = H({\bf p})$ where ${\bf p} = (p_{\omega })_{\omega \in {\Omega }}$ and $p_{\omega } = {\mathbb P}({\bf X} = \omega )$.

We will use the following entropy inequality of Shearer 
(see \cite{cfgs} or Chapter 15 \cite{aands}). We say that a family of sets 
${\cal S} = \{S_m\}_{m\in [M]}$ forms a $D$-cover of $[n]$ if every $j\in
[n]$ appears in at least $D$ sets from ${\cal S}$. 

\begin{thm}
	\label{thm: Shearer}
	Let ${\bf X} = ({\bf X}_i)_{i\in [n]}$ be a random variable taking 
	values in a finite set ${\Omega}$ and let ${\bf X}_S$ denote the random variable ${\bf X}_S = (X_i)_{i\in S}$ for all $S \subset [n]$. Then given a $D$-cover ${\cal S}$ of $[n]$, we have $\sum _{S \in {\cal S}} H({\bf X}_{S}) \geq D. H({\bf X})$.
\end{thm}

With this notation in place, we can state the partitioning lemma.

\begin{lem}
	\label{lemma: sectional control}
	Let $\A \subset \{0,1\}^n$ with $|\partial _e(\A )| \leq 
	|\A |\big ( \log _2 (2^n/ |\A |) + K \big )$.
        Suppose $I_1\cup \cdots \cup I_M = [n]$ is a partition. Then
	\begin{enumerate}[label = (\roman*)]
		\item $\sum _{m\in [M]} H(\alpha ^{I_m}) - (M-1)H(\alpha ^\es) \leq K$; 
		\item $\sum _{m\in [M]} K^{I_m} \leq K.$
	\end{enumerate} 
\end{lem}

\begin{proof}
As $\A ^{I}_{y} \subset \{0,1\}^{I}$, by Theorem \ref{thm: edge iso} we have $|\partial ^{I}_e (\A ^{I}_{y})| = |\A ^{I}_{y}| \big (\log \big (2^{|I|}/|\A ^{I}_{y}|\big ) + K_{y}^{I} \big )$ with $K^I_y \geq 0$.  Expanding this expression, we find
\begin{equation*}
	|\partial ^{I}_e (\A ^{I}_{y})|
			= 
	\alpha ^I_{y} |\A |  \log _2 \big (2^{|I|}/|\A |\big )  
	- \alpha ^I_{y}|\A | \log _2 (\alpha ^I_{y}) + \alpha ^I_{y}  K_{y}^I|\A |.
\end{equation*}
Summing over ${y}\in \{0,1\}^{J}$, as $\sum _{{y}\in \{0,1\}^{J}} 
\alpha ^I_{y} = 1$ we obtain 
\begin{equation*}
	|\partial ^{I}_e (\A )| 
	 =  \sum _{{y}\in \{0,1\}^{J}} |\partial ^{I}_e (\A ^I_{y})|
	 =  |\A| \log _2 \big (2^{|I|}/|\A | \big ) + |\A | \big (H(\mathbf{\alpha }^I) 
	 + K^I \big ).
\end{equation*}
Apply this equality for $I_1,\ldots ,I_M$. Using $|\partial _e (\A )| = \sum _{m\in [M]}|\partial ^{I_m}_e (\A )| $, we obtain 
\begin{equation*}
	|\partial _e (\A )| 
	 =  |\A | \log _2\big (2^{n}/|\A | \big ) + |\A | \Big (\sum _{m\in [M]}H(\mathbf{\alpha }^{I_m}) - (M-1) \log _2 |\A |
	 + \sum _{m\in [M]} K^{I_m} \Big ).
\end{equation*}
Since $\log _2 |\A | = H(\alpha ^\es)$ and $|\partial _e(\A )| \leq |\A | \big (\log _2 (2^n/|\A |) + K \big )$ this gives 
\begin{equation}
	\label{expression for the I-edge boundary contributions}
	\sum _{m\in [M]}H(\mathbf{\alpha }^{I_m}) - (M-1) H(\alpha ^\es)
	 + \sum _{m\in [M]} K^{I_m} 
	 	\leq 
	 K.
\end{equation}

Both $(i)$ and $(ii)$ now follow from \eqref{expression for the I-edge boundary contributions}. Indeed, $(i)$ holds since $K^{I_m} \geq 0$ for all $m\in [M]$. To see $(ii)$, by \eqref{expression for the I-edge boundary contributions} it suffices to show $\sum _{m \in [M]} H(\alpha ^{I_m})\geq (M-1) H({\alpha }^\es)$. To see this, consider selecting $x \in \A $ uniformly at random, and for all $i\in [n]$ let ${\bf X}_i$ denote the random variable given by ${\bf X}_i(x) = x_i$. For all $I' \subset [n]$ also let ${\bf X}_{I'} = ({\bf X}_i)_{i\in I'}$. Then ${\bf X}_{J_m}$ satisfies 
${\mathbb P}({\bf X}_{J_m} = y_m) = \alpha ^{I_m}_{y_m}$ for all $m\in [M]$, giving $H({\bf X}_{J_m}) = H(\alpha ^{I_m})$. Furthermore 
$H({\bf X}_{[n]}) = \log _2 |\A |$. However $\{J_m\}_{m\in [M]}$ forms a $(M-1)$-cover for $[n]$. Theorem \ref{thm: Shearer} therefore gives $\sum _{m\in [M]} H(\alpha ^{I_m}) = \sum _{m\in [M]} H({\bf X}_m) \geq (M-1) H({\bf X}_{[n]}) = (M-1) H(\alpha^\es)$. This completes the proof of the lemma.\end{proof}

%
%
%
%
%
%
%

\section{Approximate product structure}

In this section we will use Lemma \ref{lemma: sectional control} $(i)$ 
with $M = 2$ to show that if $\A $ has small isoperimetric excess 
then it has an approximate product structure
with respect to any partition $[n] = I \cup J$,
in the sense that for most elements $x \in A$
the product of `orthogonal sections' $|\A ^J_{{x}_I}||\A ^I_{{x}_J}|$
is comparable with $|\A|$.

Recall that for $y \in \{0,1\}^J$ we let $\alpha ^I_{y} = |\A ^I_y| / |\A |$, 
for $z \in \{0,1\}^I$ we let $\alpha ^J_{z} = |\A ^J_z| / |\A |$,
and $\alpha ^\es_x = 1/|\A |$ for all $x\in \A$ and $0$ otherwise. 
We also let $\alpha ^I, \alpha ^J$ and $\alpha ^\es$ 
denote the corresponding probability distributions.
The quantity $H(\alpha ^I) + H(\alpha ^J) - H(\alpha ^\es) $
can be viewed as the mutual information of the random variables
${\bf X}_I$ and ${\bf X}_J$ considered in the previous section.
If the mutual information were zero, then the variables would 
be independent, and $\A$ would have a product structure.
The following lemma can be viewed as a stability version 
of this observation.

\begin{lem} 
	\label{control of the conditional entropy implies product-like sizes}
	Let $K , \varepsilon >0$ and suppose $H(\alpha ^I) + H(\alpha ^J) - 
	H(\alpha ^\es) \leq K$. Then for at least $(1-\varepsilon )|\A |$ elements ${x} \in \A $ we have $|\A ^J_{{x}_I}||\A ^I_{{x}_J}| \geq  {|\A |}/{e.2^{K/\varepsilon}}$.
\end{lem}

\begin{proof}
Write $b_{x} = {\alpha ^\es_{x}} / ({\alpha ^I_{{x}_{J}} \alpha ^J_{{x}_{I}}})$
and let $f(t):= t \log _e t + 1 - t$. We claim that
\begin{align*}
	\log _e 2 \times  \big (H(\alpha ^I) + H({\alpha }^J) - 
	 H({\alpha }^\es) \big ) 
	 	& = 
	 \sum _{{x} \in \{0,1\}^{[n]}} 
	 \alpha ^I_{{x}_{J}} \alpha ^J_{{x}_{I}} \big (b_{x} \log _e b_{x} \big ) 
	  =  \sum _{{x} \in \{0,1\}^{[n]}} 
	 \alpha ^I_{{x}_{J}} \alpha ^J_{{x}_{I}} f(b_{x}).
\end{align*}
To see this, first note that
\begin{equation*}
	H(\alpha ^J) = 
	- \sum _{{y} \in \{0,1\}^J} \alpha ^I_{y} 
	\log _2 \big ( \alpha ^I_{y} \big ) = - \sum _{{x} \in \{0,1\}^\es } 
	\alpha ^\es_{x} \log _2 \big ( \alpha ^I_{{x}_J} \big ).
\end{equation*}
Using the analogous expressions for $H(\alpha ^I)$ and $H({\alpha }^\es)$, 
we obtain
\begin{eqnarray*}
	H(\alpha ^I) + H({\alpha }^J) - 
	 H({\alpha }^\es)
	  =  \sum _{{x}\in \{0,1\}^{[n]}} \alpha ^\es_{x} 
	 \log _2 (b_{ x}) 
	  =  \sum _{{x} \in \{0,1\}^{[n]} } \alpha ^I_{{x}_{J}} \alpha ^J_{{x}_{I}} 
	  b_{{x}}  \log _2 b_{x}.
\end{eqnarray*}	
This gives the first equality of the claim. The second follows as
\[\sum _{{x}\in \{0,1\}^{[n]} }\alpha ^I_{{x}_{J}}\alpha ^J_{{x}_{I}} 
= \big ( \sum _{{z}\in \{0,1\}^{I}} \alpha ^J_{z} \big ) 
\big ( \sum _{{y} \in \{0,1\}^{J}} \alpha ^I_{y}  \big ) = 1, \text{ and } \]
\[\sum _{{x} \in \{0,1\}^{[n]}} \big ( \alpha ^I_{{x}_{J}}\alpha ^J_{{x}_{I}} \big ) b_{x} = 
\sum _{{x} \in \{0,1\}^{[n]}} \alpha ^\es_{x} = 1.\]

Now consider $\A _D := \{{x} \in \A : b_{x} \geq D\} \subset \A $ for $D>1$.
We have
\begin{align*}
	\frac {|\A_D|}{|\A |} 
		& =  
	\sum _{{x} \in \A _D} \alpha ^\es_{x}
		=  
	\sum _{{x}\in \A : b_{x} \geq D} 
			\alpha ^I_{{x}_J} \alpha ^J_{{x}_I} b_{x} 
	 	= 
	 \sum _{{x}\in \A : b_{x} \geq D} 
	 		\alpha ^I_{{x}} \alpha ^J_{{x}_I} f(b _{x})
	 		\Big ( \frac {b_{x}} {f (b_{x})} \Big )\\	 
	 & \leq  
	 \Big ( \frac {D} {f (D)} \Big ) \times \sum _{{x}\in \A } 
	 		\alpha ^I_{{x}} \alpha ^J_{{x}_I} f(b _{x})
	 = \Big ( \frac {D} {f (D)} \Big ) \times 
	     \log _e 2 \times \big (H(\alpha ^I) + H(\alpha ^J) - 
	 H(\alpha ^\es) \big )\\
	& \leq \frac {(\log _e 2)K}{\log _e D - 1}.
\end{align*}
The first inequality holds as $f(t) \ge 0$ for $t>0$
and $g(t) := f(t)/t$ satisfies $g'(t) = t^{-1} - t^{-2} \ge 0$ for $t \ge 1$.
The following equality holds by the claim, and then the
final inequality holds since $D/f(D) \leq 1/(\log _e D -1)$ 
and $H({\alpha }^I) + H({\alpha }^J) - 	H({\alpha}^\es) \leq K$
by Lemma \ref{lemma: sectional control}(i). 
Setting $D = e2^{K/\varepsilon}$ gives $|\A _D| \leq \varepsilon |\A |$. Since $|\A |/(|\A _{{x}_I}^J||\A _{{x}_J}^I|) = b_{x} \leq D$ for all ${x} \notin \A _D$, this completes the proof.
\end{proof}

\section{Sparse sections}

In this section we prove the following result,
which shows that if $\A \subset V(Q_n)$ is sparse, 
then this is also true of typical sections of $\A$.
Another way to interpret the result 
(which is also convenient for the proof)
is to consider a random element $x$ of $\A$,
reveal all but $d$ of its coordinates,
then sample a new element of $\widetilde {x} \in \{0,1\}^{n}$
that agrees with the revealed coordinates.
Then there is typically still some uncertainty
as to whether $\widetilde {x}$ is in $\A$
(It Ain't Over Till It's Over).

\begin{lem}
\label{control of small cube sizes}
Let $\A \subset V(Q_n)$ with $|\A| = \alpha 2^n$ 
and $d \in \mathbb{N}$ with $d \le 0.15n$. Independently select:
	\begin{itemize}
		\item $x \in {\cal A}$ uniformly at random, 
		\item $I \subset [n]$ with $|I| = d $, uniformly at random.
	\end{itemize}
Then ${\mathbb E}_{x,I} (|\A ^I_{x_J}|) \leq 2\alpha ^{d/8n}2^d$, 
where $J=[n] \setminus I$. 
\end{lem}

Note that the exponent of $\aA$ is tight up to a constant factor
(for example, when $\A$ is a subcube).

\begin{proof} 
Given $x$ and $I$, we also select 
$\widetilde {x} \in \{0,1\}^{n}$ uniformly at random 
subject to ${\widetilde {x}}_{J} = {x}_J$.
Note that
\begin{align*}
	{\mathbb E}_{x,I} 
			\big (|\A ^I_{{x}_{J}}| \big ) 
	= 
	{\mathbb E}_{x,I} 
			\big ( {{\mathbb P}(\widetilde {x} \in \A |x,I)} \cdot 2^{d} \big ) 
	& = 
	 { {\mathbb P}(\widetilde {x} \in \A ) }
		  \cdot  2^d
\end{align*}
The lemma is thus equivalent to showing that 
${\mathbb P}(\widetilde {x}\in {\A}) \leq 2 \alpha ^{d/8n}$. 

To see this, we note that given 
	$w \in \A$ and $\widetilde w \in \{0,1\}^n$ with 
	$d_H({w},\widetilde {w}) = \ell$, we have
\begin{align*}
		{\mathbb P}\big (\widetilde {x} = \widetilde {w}| {x} = w \big ) 
			& = 
		{\mathbb P} \big (d_H({x},\widetilde {x}) = \ell |{x} = w\big )\cdot 
		{\mathbb P}\big ( \widetilde {x} = \widetilde {w}|
		{x}=w, d_H(\widetilde {x},{x}) = \ell \big ) \\
			& =
		\left\{ \begin{array}{ll}
         2^{-d}\binom {d}{\ell}. 
		\binom {n}{\ell }^{-1} & \mbox{if $\ell \leq d$};\\
        0 & \mbox{otherwise.}\end{array} \right. 
\end{align*}
	Let $S$ denote the linear operator $S = 2^{-d} \sum _{\ell = 0}^d 
	\binom {d}{\ell } S_{\ell }$, with $S_{\ell }$ 
as in \eqref{equation: hypercontractive operator}.
Let ${1}_{\A }$ denote the indicator function of $\A $.
Then for $w \in \A$ we have 
	\begin{equation*}
		S\big ( 1_{\A} \big ) ({w}) 
				= 
		\bigg ( \frac {1}{2^{d}} 
		\sum _{\ell =0}^{d} \binom {d}{\ell} S_{\ell} \bigg )
		 1_{\A} ({w} ) 
				=
		{\mathbb P}\big (\widetilde {x} \in \A| x = w\big ).
	\end{equation*}

We deduce that ${\mathbb P}(\widetilde {x}\in \A)
= \sum_{w \in \{0,1\}^n} {\mathbb P}(x=w)
{\mathbb P}\big (\widetilde {x} = \widetilde {w}| {x} = w \big )
= \langle \alpha ^{-1} 1_{\A},S 1_{\A} \rangle$.

Separating $S$ and 
	using the Cauchy-Schwarz inequality gives
	\begin{equation}
		\label{equation: cauchy-schwarz and hypercontractive}
		{\mathbb P}({\widetilde {x}}\in \A) 
		\leq 
		\alpha ^{-1} \sum _{\ell =0}^{d} 
		{\binom {d}{\ell }}{2^{-d}}\|1_{\A}\|_2 \|S_{\ell} 1_{\A}\|_2.
	\end{equation}
	However, by Theorem \ref{thm: Polyansky} we have 
	$\| S_{\ell }(1_{\cal A})\|_2 \leq 
	2\|{\bf 1}_{\cal A}\|_{1 + (1-2\ell / n)^2}$ for $\ell \leq d$.
As \[(1 + (1-2\ell /n)^2)^{-1} 	- {1}/{2} 
	= ({2 \ell /n - 2 \ell ^2 /n^2})({1 + (1 -2\ell /n)^2})^{-1} \geq 
	\ell/n - \ell ^2/n^2 \geq \ell /2n,\] 
since $\ell \leq n/2$, this gives  
	\begin{equation*}
		\|1_{\A}\|_2 \|S_{\ell} 1_{\A}\|_2 
		\leq 2\|1_{\A}\|_2 \|1_{\A}\|_{1 + (1-2\ell /n )^2} = 
		2\alpha ^{\frac {1}{2} + (1 + (1-2\ell /n)^2)^{-1}}
		\leq 2\alpha ^{ 1 + \ell /2n}.
	\end{equation*}
	Combined with \eqref{equation: cauchy-schwarz and hypercontractive} 
	this gives
	\begin{equation*}
		{\mathbb P}(\widetilde {x}\in \A) 
		\leq 
		\frac {2}{{2^{d}}} \sum _{\ell =0}^{d} 
		\binom {d}{\ell } \alpha ^{\ell /2n}
		= 2 \bigg (\frac {1 + \alpha ^{1/2n}}{2} 
		\bigg )^{d}.
	\end{equation*}
	To simplify, let $\alpha = e^{-L}$. As $e^{\gamma } 
	\leq 1 +  \gamma /2$ for $\gamma  \in [-1,0]$ and $1+\gamma  \leq e^{\gamma }$ for all $\gamma \in {\mathbb R}$, we find 
	\begin{eqnarray*}
		\bigg (\frac {1 + \alpha ^{1/2n}}{2} \bigg )^{d} 
			=
		\bigg (\frac {1 + e^{-L/2n}}{2} \bigg )^{d}
			\leq 
		\bigg (\frac {1 + (1 - L/4n)}{2} \bigg )^{d}
			=
		\bigg ( 1 - \frac {L}{8n} \bigg )^{d} 
			\leq 
		e^{-d L/8n} = \alpha ^{d/8n}.
	\end{eqnarray*}
	Therefore $ {\mathbb P}( \widetilde {x}\in \A) \leq 
		2\alpha ^{d/8n}$, completing the proof of the lemma.
\end{proof}

\section{Finding a coordinate of large influence}

In this section we prove that if ${\cal A}\subset \{0,1\}^n$ 
has small isoperimetric excess and is not close to being
the whole cube then there is a coordinate of large influence. 

\begin{thm}
	\label{coordinate of large influence}
		Let ${\A } \subset V(Q_n)$ with $|\A | \leq \big ( \frac{7}{8} \big )2^{n}$ and 
		$|\partial _e (\A ) | \leq |\A | \big ( \log _2 ( 2^n / |\A | ) + K \big )$. 
		Then $\max _{i\in [n]} I _i ({\mathbf 1}_{\A} ) \geq 2^{-C(K+1)^2}{|\A |}/{2^n}$, 
		for some constant $C>0$.
\end{thm}

\begin{proof}
Let $|\A | = \alpha 2^n$, where $\alpha \leq 7/8$, and let $\beta _i = I_i({\mathbf 1}_{\cal A})$ for all $i\in [n]$ and $\beta = \max _{i\in [n]} \beta _i$. We will also fix a number of parameters to be used in the proof. Set $c_0 = \min \{\varepsilon _0, 1/8\}$, where $\varepsilon _0$ is as in Theorem \ref{thm: Ellis}. Also set $C_1 = 2^{12}/c_0$, $\delta = c_0/(32(K+1))$ and $M = \lceil 1/4\delta \rceil $. Lastly set $C = 32C_1/c$, with $c$ as in Theorem \ref{thm: Talagrand}.

We first consider the case when $\alpha \geq 2^{-C_1(K+1)^2}$, where the result follows from
Talagrand's inequality. Indeed, in this case 
$|\partial _e(\A )| \leq |\A| \big (C_1(K+1)^2 + K \big) \leq 2C_1(K+1)^2|{\cal A}|$. 
As $\aA \le 7/8$, we have ${\mathbb V}ar ({\bf 1}_{\cal A}) \geq |{\cal A}|/2^{n+3}$, 
so Theorem \ref{thm: Talagrand} gives
	\begin{equation*}
		\frac {c|{\cal A}|}{2^{n+3}} 
			\leq 
		c.{\mathbb V}ar ({\bf 1}_{\cal A}) 
			\leq 
		\sum _{i\in [n]} \frac {\beta _i}{1 - \log _2 \beta _i}			
			\leq 
		\frac {I({\bf 1}_{\cal A}) }{\log _2 ( \beta ^{-1})} 
			=
		\frac {|{\partial }_e({\cal A})|}{2^{n-1} \log _2 (\beta ^{-1})} 
			\leq 
		\frac {2C_1(K+1)^2|{\cal A}|}{2^{n-1} \log _2 (\beta ^{-1} )}.
	\end{equation*}
Rearranging $\beta \geq 2^{-32C_1(K+1)^2/c} = 2^{-C(K+1)^2}$, as required.

It remains to consider the case $\alpha \leq 2^{-C_1(K+1)^2}$.
We start by giving an overview of the argument in this case.
We will find a partition $I \cup J$ of $[n]$
so that for many elements $x$ of $\A$ 
the $I$-section is sparse and has small isoperimetric excess,
and the product of orthogonal sections 
through $x$ is comparable with $\A$.
We can then apply Ellis' theorem
to find a subcube ${\mathcal C} \subset \{0,1\}^{I}$
such that many elements of $\A$ 
have an $I$-restriction in ${\mathcal C}$.
Finally, we show that one of the coordinates
that is influential for ${\mathcal C}$ 
must also be influential for $\A$.

To begin, select a partition $[n] = I_0 \cup I_1 \cup \ldots \cup I_M$ uniformly at
random, with $|I_m| = d = \lceil \delta n \rceil $ for all $m\in [M]$ and $|I_0| = n - Md$.
(This is possible as $n \geq C_1(K+1)^2$, $\delta n \geq C_1c_0(K+1)/32 \geq 1$
and $Md \le (2\delta )^{-1}(2\delta n) \le n$.) 
Write $J_m = [n] \setminus I_m$ for all $m\in [M]$. 
We say $I_m$ is \emph{controlled} if 
	\begin{equation*} 
		\big | \big \{ x \in \A : |\A ^{I_m}_{{x} _{J_m}}| 
			\leq 16 \alpha ^{d /8n}2^{d} \big \} \big | 
		\geq 
		{3|\A |}/{4}.
	\end{equation*}
By Lemma \ref{control of small cube sizes} we have 
	\begin{equation*}
		\frac {1}{4} \times {\mathbb P}(I_m\mbox{ is not controlled}) \times 16 \alpha ^{d /8n}2^{d} 
				\leq 
		{\mathbb E}_{x,I_m}(|\A ^{I_m}_{x_{J_m}}|) \leq 2\alpha ^{d/8n} 2^d.
	\end{equation*} 
This gives ${\mathbb P}(I_m\mbox{ is controlled}) \geq 1/2$.
Letting $S_1 = \{m\in [M]: I_m \mbox{ is controlled}\}$, we find that $\mathbb {E}(|S_1|) \geq M/2$. Fix a choice of $I_1, \ldots ,I_M$ such that $|S_1| \geq M/2$. 

Now set $S_2 = \{ m\in [M]: {K}^{I_m} \leq 4K/M\}$. As $\sum _{m\in [M]} {K}^{I_m} \leq K$ by Lemma \ref{lemma: sectional control} $(ii)$, Markov's inequality gives $|S_2| \geq 3M/4$. 
Combined with the previous paragraph, this gives $S_1 \cap S_2 \neq \emptyset $. Fix $m\in S_1 \cap S_2$ and take $I = I_m$ and $J  = J_m$.

To proceed we now consider the following subsets of $\A $:
	\begin{enumerate}
		\item [(i)] $\A _1 =\big \{ {x}\in \A : |\A ^{I}_{{x}_{J}}|  
								\leq  16 \alpha ^{d /8n}2^{d} \big  \}$;
		\item [(ii)] $\A _2 = \big \{{x}\in \A : K^{I}_{{x}_{J}} \leq 16K/M 
		\big \}$;
		\item [(iii)] $\A _3 = \big \{{x}\in \A : 
		|\A ^{I}_{{x}_{J}}||\A ^{J}_{{x}_{I}}| \geq |\A |/e 2^{4K} 
		\big \}$.
	\end{enumerate}
Further let ${\cal B} = \A _1 \cap \A _2 \cap \A _3$. We claim that $|{\cal B}| \geq |\A |/4$. To see this, first note that $|\A _1| \geq 3|\A |/4$ as $m\in S_1$. Also, since ${\mathbb E}_{{x}\sim \A } K^{I}_{{x}_{J}} = {K}^{I} \leq {4K}/{M}$ as $m\in S_2$, by Markov's inequality $|\A _2| \geq 3|\A |/4$. Lastly, applying Lemma \ref{control of the conditional entropy implies product-like sizes} with $\varepsilon = 1/4$ gives $|\A _3| \geq 3|\A |/4$. Therefore $|{\cal B}| \geq |\A |/ 4$ as claimed.

We will now show that some $x \in {\cal B}$ has a coordinate of large influence in ${\cal A}^{I}_{x_{J}} \in \{0,1\}^I$ which extends to give large influence for ${\cal A}$. To see this, note that partitioning ${\cal B}$ over the $I$-sections gives
\begin{equation*}
	\sum _{{y}\in \{0,1\}^{J}} |{\cal B} ^{I}_{y}| = |{\cal B}| 
	\geq 
		\frac{|\A |}{4} = 
	\sum _{{y}\in \{0,1\}^{J}} \frac{|\A ^{I}_{y}|}{4}.
\end{equation*}
Therefore $|{\cal B} ^{I}_{{y}_0}| \geq |\A ^{I}_{{y}_0} |/4 > 0$ for some ${y}_0\in
\{0,1\}^{J}$. Fixing any ${x} \in {\cal B} ^{I}_{{y}_0}$, as $x \in \A_2$
we have $K^{I}_{{y}_0} = K^{I}_{{x}_{J}} \leq 16K/M \leq c_0$, i.e.\
\begin{align*}
|\partial _e^{I}(\A ^{I}_{{y}_0})| 
	\leq
|\A ^{I}_{{y}_0}|\big (\log _2 (2^{|I|}/|\A ^{I}_{{y}_0}|) + c_0 \big ) .
\end{align*}
Theorem \ref{thm: Ellis} therefore gives $|\A ^{I}_{{y}_0} \triangle {\mathcal C}| \leq 3 c_0 |\A ^{I}_{{y}_0}|/\log _2(c_0^{-1}) \leq |\A ^{I}_{{y}_0}|/8$ for some subcube ${\mathcal C} \subset \{0,1\}^{I}$. As $|{\cal B} \cap \A ^{I}_{{y}_0}| \geq |\A ^{I}_{{y}_0} |/4$, we also find $|\A ^{I}_{{y}_0} \cap {\cal B} \cap {\mathcal C}| \geq |\A ^{I}_{{y}_0}|/8$. 

Set ${\mathcal D} = \{{x} \in \A : {x}_{I} \in {\mathcal C}\}$. Note that ${\cal D}$ is insensitive to coordinates in $J$, in the sense that if $x \in {\cal D}$ and $\widetilde {x} \in \A $ with $x_I = {\widetilde x}_I$ then $\widetilde {x} \in {\cal D}$. Therefore $\bigcup _{{z} \in {\mathcal C}} \A ^{J}_{z} \subset {\mathcal D}$. However given ${x} \in {\cal A}^I_{y_0} \cap {\cal B} \cap {\cal C}$ we have $x \in \A _3$, so $
	|\A ^{I}_{{y}_0}||\A ^{J}_{{x}_{I}}|  = |\A ^{I}_{{x}_{J}}||\A ^{J}_{{x}_{I}}| \geq \frac {|\A |}{e2^{4K}}$. Combining this gives
\begin{align}
	\label{equation: D control}
	|{\mathcal D}| 
		 \geq 
	\sum _{{x}\in {\cal A}^I_{y_0} \cap {\cal B} \cap {\cal C}} 
	|\A ^{J}_{{x}_{I}}|
		 \geq 
	\sum _{{x}\in {\cal A}^I_{y_0} \cap {\cal B} \cap {\cal C}} 
	\frac {|\A |}{e2^{4K} |\A ^{I}_{{y}_0}|} 
		& \geq 
	|{\cal A}^I_{y_0} \cap {\cal B} \cap {\cal C}| \cdot 
		\frac {|\A |}{e2^{4K} |\A ^{I}_{{y}_0}|} \nonumber \\
		& \geq 
	\frac {|\A ^{I}_{{y}_0}|}{8} \cdot 
	\frac {|\A |}{e2^{4K} |\A ^{I}_{{y}_0}|} \geq 
	\frac {|\A |}{2^{4K + 5}}.
\end{align}
Thus a large proportion of elements ${x} \in{\mathcal A}$ satisfy ${x}_{I} \in {\mathcal C}$. 

We will now show that one of the coordinates that are influential for ${\cal C} \subset \{0,1\}^I$ must also be influential for ${\cal A}$. To see this, as ${\mathcal C} \subset \{0,1\}^{I}$ is a subcube there is $T = \{i_1,\ldots , i_t\} \subset I$ and ${z}_0 \in \{0,1\}^{T}$ with ${\mathcal C} = \{ {z} \in \{0,1\}^{I}: {z}_{T} = {z}_0\} $ and $ \log _2|{\mathcal C}|  = d - t = |I \setminus T|$. As $|{\cal C}| \leq 
|{\cal A}^I_{{y}_0}| + |{\cal A}^I_{{y}_0} \triangle {\cal C}| 
\leq 2 |{\cal A}^I_{{y}_0}| \leq 2^5 \alpha ^{d / 8n} 2^d$, we find 
\begin{equation}
	\label{equation: dimension control}
	t 
			= 
	d - \log_2 |{\cal C}| 
			\geq 
	\frac {d}{8n} \log _2 ( \alpha ^{-1}) - 5 
			> 
	\frac{ \delta }{8} \log _2 ( \alpha ^{-1} ) - 5 
			\geq
	\frac{c_0C_1(K+1)}{2^8}  - 5 
			\geq 
	4K+7.
\end{equation}
Here we used $\alpha \leq 2^{-C_1(K+1)^{2}}$, $\delta = c_0/(32(K+1))$ and $C_1 = 2^{12}/c_0$. 

Finally, suppose for a contradiction that
$\beta _{i} \leq |\A |/ (2^{4K + 6} 2^n)$ for all $i\in T$.
For $0 \le \ell \le t$ let 
	\begin{equation*}
		\A _{\ell } 
			= 
		\A \cap  \big \{ {x} \in \{0,1\}^n: {x}_{\{i_{\ell + 1},\ldots ,i_{t}\}} 
			= 
		{z}_{\{i_{\ell + 1},\ldots ,i_{t}\}} \big \}.
	\end{equation*} 
Clearly ${\cal D} = {\cal A}_0 \subset \A _1 \subset \cdots \subset \A _t = \A $. 
As ${\beta }_{i_{\ell }} \leq |\A |/2^{4K +6}2^n$, 
we have $|\A _{\ell }| \geq 2|\A _{\ell -1}| - {|\A |}/{2^{4K + 6}}$. Equivalently
$|\A _{\ell }| - {|\A |}/{2^{4K + 6}} \geq 2 \big ( |\A _{\ell -1}| - {|\A |}/{2^{4K + 6}} \big )$.
Taking $\ell = t$, we find 
	\begin{align*}
		|\A | 
				>
		|\A _{t}| - |\A |/2^{4K +6}
				\geq 
		2^t (|\A _{0}| - |\A |/2^{4K +6} ) 
				\geq 
		2^{t - 4K - 6} |\A |.
	\end{align*}
The final inequality holds since $|\A _0|  = |{\cal D}| \geq |\A | / 2^{4K + 5}$ by \eqref{equation: D control}. However, as $t \geq 4K + 7$ from \eqref{equation: dimension control}, this is a contradiction, and so, as $C_1 \geq 6$ we have $\beta \geq \max _{i\in T} \beta _i \geq 2^{-4K - 6}|\A |/2^{n} \geq 2^{-C_1(K+1)^2}|{\cal A}|/2^n$, as claimed. \end{proof}

\section{Almost isoperimetric sets are close to a union of cubes}

With Theorem \ref{coordinate of large influence} in hand, we can now prove Theorem \ref{dimension free stability}. 

\begin{proof}[Proof of Theorem \ref{dimension free stability}]
To begin, let $C_1 \geq 1$ be the constant given by Theorem \ref{thm: Friedgut} and let $C_2\geq 1$ be the constant given in Theorem \ref{coordinate of large influence}. Set $C = \max \{6{C}_1, 8C_2\}$ and let $g: {\mathbb R} \to {\mathbb R}$ be the function $g(x) = 2^{2^{C(x+1)^2}}$. We will show that for all $K\geq 0$ and $\varepsilon > 0$, given a set ${\cal A} \subset \{0,1\}^n$ with $|{\partial }_e({\cal A})| \leq |{\cal A}|\big ( \log _2(2^n/|{\cal A}|) + K \big )$ there are disjoint cubes $C_1,\ldots ,C_{L}$ with $|{\cal A} \triangle (\cup _{C \in {\cal C}} C )| \leq \varepsilon |{\cal A}|$ and $L \leq g(K/\varepsilon )$. 

Before beginning the proof, we note that this seemingly weaker bound on $L$ implies the bound stated in Theorem \ref{dimension free stability}. Indeed, if $\varepsilon >1$ then ${\cal A}$ can be ${\varepsilon }$-approximated $0$ subcubes. For $\varepsilon \leq 1$, if $K \leq K/\varepsilon < c_0:= \max \{\varepsilon _0, 1/8\}$ then ${\cal A}$ can $\varepsilon $-approximated by $1 \leq L(K/\varepsilon )$ subcube by Theorem \ref{thm: Ellis}. Otherwise, $1 \leq c_0^{-1} K/\varepsilon $ and $g(K/\varepsilon ) \leq L(K/\varepsilon ):= 2^{2^{4Cc_0^{-2}(K/\varepsilon )^2}}$.

We will prove the result by induction on $|\A | + n$.  
Clearly it holds when $|{\cal A}| = 1$ for all $n$. 
We also claim that the result holds when $|\A | \geq \big ( \frac{7}{8} \big ) 2^n$. 
Indeed, in this case we consider ${\cal A}^{c} = \{0,1\}^n \setminus {\cal A}$,
and write $|{\cal A}^c| = \alpha 2^n$ so that $\alpha \leq \frac{1}{8}$. Using $1-x \geq 2^{-2x}$ for $x \in [0,1/8]$ and applying Theorem \ref{thm: edge iso} to ${\cal A}^c$ we find
	\begin{equation*}
		2^n \big ( \log _2 ( 1 - \alpha )^{-1} + K \big ) 
					\geq 
		|\partial _e({\cal A})| 
					= 
		|\partial _e({\cal A}^c)| 
					\geq 
		\alpha 2^n \log _2 (\alpha ^{-1}) 	
					\geq 
		3\alpha 2^n \geq \Big ( \frac {3}{2} \Big ) 2^n \log _2 (1- \alpha )^{-1}.
	\end{equation*}
Thus $K \geq \frac {1}{2} \log _2 (2^n/|{\cal A}|)$ and $|\partial _e(\A )| \leq 3K2^n$. By Theorem \ref{thm: Friedgut} there are disjoint subcubes $C_1,\ldots ,C_L$ such that $|\A \triangle (C_1 \cup \cdots \cup C_L)| \leq \big ( \frac {\varepsilon }{2} \big ) 2^n \leq \varepsilon |\A |$ with $L \leq 2^{2^{C_1(3K/(\varepsilon /2))}} \leq g(K/\varepsilon )$, as desired.

Let ${\cal A} \subset \{0,1\}^n$ and assume that the result holds for smaller values of $|{\cal A}| + n$, and that $|{\cal A}| \leq \big (\frac{ 7}{8} \big )2^n$. We can apply Theorem \ref{coordinate of large influence} to ${\mathcal A}$ find a coordinate $j\in [n]$ with $I_j({\bf 1}_{\cal A}) = b_j|\A |/2^n$ where $b_j \geq c(K):= 2^{-C_2(K+1)^2}$. Without loss of generality $j = n$. 
Set $\A ^- = \A ^{[n-1]}_0$ and $\A ^+ = \A ^{[n-1]}_1$ and let $\gamma \in [0,1]$ with $|\A ^-| = \gamma |\A |$  and $|\A ^+| = (1 - \gamma ) |\A |$. By symmetry we may assume $\gamma \leq 1/2$. By Theorem \ref{thm: edge iso}, there are $K^-, K^+ \geq 0$ with
 \begin{align*}
 	|\partial _e^{[n-1]}(\A ^-)| 
 		& =
 	|\A ^-| \big ( \log _2(2^{n-1}/|\A ^-|) + K^- \big ),\\
 	|\partial _e^{[n-1]}(\A ^+)| 
 		& =
 	|\A ^+| \big ( \log _2(2^{n-1}/|\A ^+|) + K^+ \big ).
 \end{align*}
Expanding these expressions gives 
 \begin{align*}
 	|\partial _e^{[n-1]}(\A ^-)| 
 		& = 
 	\gamma |\A | \log _2 (2^n/|\A |) + \gamma |\A |(\log _2 \gamma ^{-1} - 1 + K^{-} ),\\
 	|\partial _e^{[n-1]}(\A ^+)| 
 		& = 
 	(1- \gamma )|\A | \log _2(2^n/|\A |) + 
 	(1- \gamma )|\A | (\log _2(1-\gamma)^{-1} - 1 + K^{+} ).
 \end{align*}
Combining these identities together with the contribution $b_n|\A |$ from 
the edges in direction $n$ gives 
\begin{align*}
	|\partial _e(\A )|
		& = 
	|\partial ^{[n-1]}_e(\A ^-)| + 
	|\partial ^{[n-1]}_e(\A ^+)| + 
	|\A ^- \triangle \A ^+|\\
		& = 
	|\A | \log _2 2^n/|\A | + |\A |\big ( H(\gamma ) - 1  
	+ b_n  + \gamma K^- + (1-\gamma )K^+ \big ).
\end{align*}
By possibly decreasing $K$ we can assume that
$|\partial _e (\A )| = |\A |\big (\log _2(2^n/|\A |) + K\big )$.
Then
\begin{equation}
	\label{improved isoperimetric error}
	\gamma K^- + (1-\gamma )K^+ 
			= 
	K - (H(\gamma ) -2\gamma ) - (b_n - (1-2\gamma )) 
	 := \widetilde{K}.
\end{equation}
Note that both bracketed terms here are non-negative. Indeed, $H(\gamma )$ is concave on $[0,1/2]$ as $H'(\gamma ) = \log _2 (\gamma /(1-\gamma )) \geq 0$ and so $H(\gamma ) \geq 2\gamma $ for $\gamma \in [0,1/2]$ as $H(0) = 0$ and  $H(1/2) = 1$. The second term is also non-negative as $b_n|\A | = |\A ^+ \triangle \A ^-| \geq |\A ^+| - |\A ^-| = (1- 2\gamma )|\A |$.  By partitioning the contribution of \eqref{improved isoperimetric error} we find $\delta \in [0,1]$ with $\gamma K^- = \delta \widetilde K$ and $(1 - \gamma ) K^{+} = (1-\delta ){\widetilde K}$. Also fix $E := 2^{-2C_2(K/\varepsilon +1)^2}$.

First suppose that $\widetilde {K} \leq K -E$. In this case, we will approximate both $\A ^-$ and $\A ^+$ by appropriate collections of cubes. Set $\varepsilon ^- = \delta \varepsilon /\gamma $ and 
$\varepsilon ^+ = (1 - \delta ) \varepsilon /(1 - \gamma )$. By the inductive hypothesis, there are disjoint subcubes ${\cal C}^- = \{C^-_j\}_{j\in [L^-]}$ with $L^- \leq L(K^-, \varepsilon ^-)$ and ${\cal C}^+ = \{C^+ _i\}_{i\in [L^+]}$ with $L^+ \leq L(K^+, \varepsilon ^+)$ so that 
	\begin{equation*}
		|{\cal A}^- \triangle (\cup _{C \in {\cal C}^-} C)| 
				\leq 
		\varepsilon ^- |{\cal A}^-| = \delta \varepsilon |\A |
		\qquad \mbox {and} \qquad 		
		|{\cal A}^+ \triangle (\cup _{C \in {\cal C}^+} C)| 
				\leq 
		\varepsilon ^+ |{\cal A}^+| = (1 - \delta )\varepsilon |\A |.
	\end{equation*} 
We can naturally identify cubes in ${\cal C}^-$, ${\cal C}^+$ 
with subcubes of $Q_n$ in which the $n$th coordinate is $0$, $1$, respectively.
Taking ${\cal C} = {\cal C}^- \cup {\cal C}^+$ we find $|{\A} \triangle ( \cup _{C \in {\cal C}} C)| \leq \varepsilon |\A |$. Therefore 
	\begin{equation}
		\label{equation: cube bound}
		|{\cal C}| 
				\leq 
		L(K^-,\varepsilon ^-) + L(K^+,\varepsilon ^+) 
				\leq 
		2 g({\widetilde K}/\varepsilon ) 
				\leq 
		2g\big ((K - E)/\varepsilon \big ) 
				\leq 
		2g\big (K/\varepsilon  - E \big ).
	\end{equation} 
Note that the function $h (x) = \log _2 g(x) = 2^{C(x+1)^2}$ satisfies $h'(x) \geq 2(\log _e2)C(x+1)h(x) \geq h(x)$ and so $h'(x)$ is increasing. By the mean value theorem, using $E \leq (K/\varepsilon + 1 )/2$, this gives
	\begin{equation*}
		1 + h(K/\varepsilon -E) \leq 1 + h(K/\varepsilon ) - Eh'(K/\varepsilon -E) \leq 1 + h(K/\varepsilon ) - E2^{C(K/\varepsilon + 1)^2/4} \leq h(K/\varepsilon ).
	\end{equation*} 
Here $E2^{C(K/\varepsilon + 1)^2/4}  = 2^{(C/4 - 2C_2)(K/\varepsilon +1)^2} \geq 1$ as $C \geq 8{C_2}$. 
Exponentiating, and combining with \eqref{equation: cube bound}
we find $|{\cal C}| \leq 2 g(K/\varepsilon - E) \leq g(K/\varepsilon )$, completing the proof of this case.

It remains to deal with the case $\widetilde {K} \geq K - E$. We claim that this is only possible if $\gamma \leq E$. To see this, note that by \eqref{improved isoperimetric error} in this case we have ($i$) $b_n - (1-2\gamma ) \leq E$ and ($ii$) $H(\gamma) - 2\gamma \leq E$. Since $b_n \geq c(K) \geq 2E$, by ($i$) we have $\gamma \leq 1/2 - c(K)/4$. Also $H(\gamma ) - 2\gamma \geq 2\gamma - 4 \gamma ^2 = 
2\gamma (1 -2 \gamma ) \geq \min (\gamma ,1-2\gamma )$ for $\gamma \in [0,1/2]$. Therefore $H(\gamma ) - 2\gamma > E$ for $\gamma \in (E, 1/2 - c(K)/4]$, which by ($ii$) forces $\gamma \leq E$, as claimed.

As $\A ^-$ is small, we can approximate $\A$ by deleting $\A ^-$
and approximating $\A ^+$ by subcubes
with accuracy $\varepsilon ' = (\varepsilon - \gamma ) / (1 - \gamma )$. 
By induction, there are disjoint cubes ${\cal C} = (C_i)_{i \in [L]}$ with 
$|\A ^+ \triangle (\cup _{C \in {\cal C}} C)| \leq \varepsilon '|\A ^+|$ and $L \leq g(K ^+/ \varepsilon ')$. But then $|{\cal A} \triangle (\cup _{C \in {\cal C}} C)| \leq \varepsilon '|\A ^+| + |\A ^-| \leq \varepsilon |\A |$. Lastly, using $\gamma \leq E$, $C_2 \geq 1$ and 
$K^+ \leq \big ( K - (H(\gamma ) - 2 \gamma )\big )/(1- \gamma )$ we have 
	\begin{equation*}
		\frac{K^+ }{\varepsilon '} 
				\leq 
		\frac {K - \gamma \log _2 \gamma ^{-1} /3}{\varepsilon - \gamma } 				\leq 
		\frac {K - 2\gamma C_2 (K/\varepsilon +1)^2/3}  
		{\varepsilon - \gamma }  
				\leq 
		\frac {K - \gamma K/\varepsilon }{\varepsilon - \gamma } 
				= 
		\frac{K }{\varepsilon }.
	\end{equation*}
Therefore $L \leq g({K^+/\varepsilon '}) \leq g(K/\varepsilon )$, completing the proof of this case and the Theorem.
\end{proof}	

\noindent \textbf{Note:} Keller and Lifshitz \cite{KL} have independently and simultaneously proved a stronger version of our main theorem, with an essentially optimal bound of $L(K,\varepsilon ) \leq 2^{2^{C(K/\varepsilon )}}$. Although our bounds are weaker, as our approach is significantly different we feel that the methods may be useful for similar problems in the future, particularly if they are not amenable to the compression arguments used in \cite{KL}. \vspace{2mm}

\noindent \textbf{Acknowledgement:}
We would like to thank David Ellis for bringing the independent work  of Keller and Lifshitz \cite{KL} to our attention.


\begin{thebibliography}{99}

\bibitem{aands} N. Alon and J. Spencer,
\emph{The Probabilistic Method}, Wiley, 2008.

\bibitem{Beckner} W. Beckner, Inequalities in Fourier analysis. 
\textit{Ann. Math.}, \textbf{102} (1975), 159--182.

\bibitem{BL} M. Ben-Or and N. Linial, Collective Coin Flipping, Robust Voting Schemes and Minima of Banzhaf Values, \textit{FOCS. IEEE Computer Society}, (1985), 408--416.

\bibitem{Bon} A. Bonami, \'Etude des coefficients de Fourier des fonctions de $L_p(G)$. 
\emph{Annales de l'Institut Fourier}, \textbf{20} (1970), 335--402. 

\bibitem{bernstein} A. J. Bernstein, Maximally connected arrays on the n-cube, \textit{SIAM J. Appl. Math.} \textbf{15} (1967), 1485--1489.

\bibitem{bol} B. Bollob\'as, \emph{Combinatorics}, Cambridge University Press, 1986.


\bibitem{B} J. Bourgain, appendix to \cite{FrShThres}.

\bibitem{cfgs} F.R.K. Chung, P. Frankl, R.L. Graham and J.B. Shearer, Some intersection theorems for ordered sets and graphs. \textit{J. Comb. Theory, Ser. A}, \textbf{43} (1986), 23--37.

\bibitem{ct}
T.M. Cover and J.A. Thomas,
{\em Elements of Information Theory},
Wiley, 2006.

\bibitem{Ellis} D. Ellis, Almost isoperimetric subsets of the discrete cube, \textit{Combin. Probab. Comput.} \textbf{20} (2011), 363--380.

\bibitem{EKL} D. Ellis, N. Keller and N. Lifshitz, On the structure of subsets of the discrete cube with small edge boundary, \textit{arXiv:1612.06680}.

\bibitem{FS} D. Falik and A. Samorodnitsky, Edge-Isoperimetric Inequalities and Influences, \textit{Combin. Probab. Comput.}, \textbf{16} (2007), 693--712.

\bibitem{Friedgut} E. Friedgut, Boolean functions with low average sensitivity depend on few coordinates,
\textit{Combinatorica} \textbf{18} (1998), 27--35.

\bibitem{FrShThres} E. Friedgut, Sharp thresholds of graph proprties, and the k-sat problem, \textit{J. Amer. Math. Soc.} \textbf{12} (1999), 1017--1054.

\bibitem{FK} E. Friedgut and G. Kalai, Every monotone graph property has a sharp threhold, \textit{Proc. Amer. Math. Soc.} \textbf{124} (1996), 2993--3002.

\bibitem{harper} L. H. Harper, Optimal assignments of numbers to vertices, \textit{SIAM J. Appl. Math.} \textbf{12}
(1964), 131--135.

\bibitem{hart} S. Hart, A note on the edges of the $n$-cube, \textit{Discrete Math.} \textbf{14} (1976), 157--163.

\bibitem{H} H. Hatami, A structure theorem for Boolean functions with small total influences, \textit{Ann. Math.}, \textbf{176} (2012), 509--533.

\bibitem{KK} J. Kahn and G. Kalai, Thresholds and Expectation Thresholds, \textbf{16} (2007), 495--502.

\bibitem{KL} N. Keller and N. Lifshitz, Approximation of biased Boolean functions of small total influence by DNF's, \textit{preprint}.

\bibitem{KKL} J. Kahn, G. Kalai and N. Linial, The influence of variables on Boolean functions. \textit{Proceedings of the 29th Annual IEEE Symposium on Foundations of Computer Science}
(1988), 68--80.

\bibitem{lindsey} J. H. Lindsey II, Assignment of numbers to vertices, \textit{Amer. Math. Monthly} \textbf{71} (1964), 508--516.

\bibitem{moo}
E. Mossel, R. O'Donnell and K. Oleszkiewicz,
Noise stability of functions with low influences: invariance and optimality,
\emph{Ann. Math.} {\bf 171} (2010), 295--341.

\bibitem{o'd} R. O'Donnell,
{\em Analysis of Boolean Functions},
Cambridge University Press, 2014.

\bibitem{poly} Y. Polyanskiy, Hypercontractivity of spherical averages in Hamming space, \textit{arXiv:1309.3014}.

\bibitem{R} R. Rossignol, Threshold for monotone symmetric properties through a logarithmic Sobolev inequality, \emph{Ann. Prob.}, \textbf{34} (2005), 1707-1725.

\bibitem{russo} L. Russo, An approximate zero-one law. \textit{Z. Wahrsch. Verw. Gebiete} \textbf{61} (1982), 129--139.

\bibitem{Tal} M. Talagrand, On Russo's approximate zero-one law, \textit{Ann. Probab.} \textbf{22} (1994), 1576--1587.

\end{thebibliography}
\end{document}